\documentclass[12pt]{amsart}
\usepackage{amsfonts}

\usepackage{amsxtra}
\usepackage{latexsym}
\usepackage{amssymb}
\usepackage{amsthm}
\usepackage{newlfont}
\usepackage{verbatim}
\usepackage{amsmath}
\usepackage{color}
\usepackage{enumerate}
\usepackage{mathrsfs}
\usepackage{hyperref}

\newtheorem{thm}{Theorem}[section]

\newtheorem{prop}[thm]{Proposition}
\newtheorem{rmk}[thm]{Remark}

\theoremstyle{definition}  
\newtheorem{dfn}[thm]{Definition}

\newcommand{\q}{\quad}
\newcommand{\qq}{\quad\quad}
\newcommand{\qqq}{\quad\quad\quad}

\newcommand{\nf}{\infty}

\newcommand{\al}{\alpha}

\newcommand{\ga}{\gamma}

\newcommand{\de}{\delta}

\newcommand{\ve}{\varepsilon}

\newcommand{\la}{\lambda}

\newcommand{\si}{\sigma}

\newcommand{\vp}{\varphi}

\newcommand{\rrr}{\mathbf R}

\newcommand{\rn}{\mathbf R^n}

\newcommand{\zzz}{\mathbf Z}
\newcommand{\zp}{\mathbf Z^+}

\newcommand{\ccc}{\mathbf C}

\newcommand{\wlq}{L^{q,\infty}}

\newcommand{\f}{\frac}

\newcommand{\lab}{\label}
\newcommand{\bee}{\begin{equation}}
\newcommand{\eee}{\end{equation}}

\newcommand{\R}{\mathbb{R}}

\begin{document}

\title
{Remarks on countable subadditivity}

\author{Loukas Grafakos}
\address{Department of Mathematics, University of Missouri, Columbia, MO 65211, USA}
\email{grafakosl@missouri.edu}

\author[Monica Vi\c{s}an]{Monica Vi\c{s}an}
\address{Department of Mathematics, University of California, Los Angeles, CA 90095, USA}
\email{visan@math.ucla.edu}



\begin{abstract}
We discuss how countable subadditivity  of   operators can be derived from subadditivity  under mild forms of continuity, and provide examples manifesting such circumstances.
\end{abstract}

\maketitle

\section{Introduction}
 
A central theorem in modern analysis is Hunt's interpolation theorem, \cite{Hunt64, Hunt66}.  
Arising as the culmination of  earlier  studies, e.g. \cite{Calderon60, Calderon64, KreinSemenov, LionsPeetre, Peetre, SteinWeiss58, SteinWeiss59}, Hunt's theorem provides the definitive formulation of the celebrated Marcinkiewicz interpolation theorem in the setting of Lorentz spaces. It is well known that Lorentz spaces constitute  a natural scale of spaces that contain both $L^p$ and weak $L^p$ spaces.

There are elegant and well-documented treatments of these interpolation results centered around simple functions; for textbook presentations see, for example, \cite{CFA, SW:book}.  The interpolation results are then extended from simple functions to those Lorentz spaces in which simple functions are dense.  However, simple functions are \emph{not} dense in some of the most conspicuous Lorentz spaces arising in analysis, namely, the weak $L^p$ spaces.

Due to the reliance on the density of simple functions, standard proofs of interpolation theorems often offer incomplete treatments in the case of weak $L^p$ spaces.  Inevitably, interpolation results in these spaces   require different ad-hoc arguments tailored to the particular operators appearing in each application.  This shortcoming somewhat compromises a universal formulation in the   elegant    theory of interpolation.

In this note, we advocate that the property of countable subadditivity offers an effective and universal remedy to this problem. Countable subadditivity provides a streamlined approach to proving these interpolation theorems, obviating the need for various limiting procedures imposed by the consideration of simple functions, thus allowing 
the uniform treatment of all Lorentz spaces.  Accordingly, we believe that in addition to its scientific value, this approach has a significant expository and pedagogical merit.  

In Proposition~\ref{P1}, we prove that the hypotheses of the Hunt interpolation theorem already guarantee countable subadditivity of the operator.  Indeed, we show that a subadditive operator satisfying the hypotheses of Hunt's theorem is automatically countably subadditive on the interpolated spaces.

More generally, the question we explore in this paper is under what reasonable conditions a subadditive mapping may actually be countably subadditive.  Recall that a mapping $T$ defined on a normed vector  space  $(W, |\cdot |_W)$ and taking values in another  normed vector  space $(V, |\cdot |_V)$ is called {\it subadditive} if 
\bee\lab{0}
\big| T(f+g) \big|_{V} \le \big| T(f ) \big|_{V}+\big| T( g) \big|_{V}
\eee
for every $f,g\in W$.  We say that $T$ is {\it countably subadditive} if whenever a series $\sum_{j\in \zzz} f_j$ converges in the norm of $W$, we have  
\bee\lab{1}
\Big| T\Big(\sum_{j\in \zzz} f_j \Big) \Big|_{V} \le \sum_{j\in \zzz}\big| T(f_j ) \big|_{V} .
\eee

Many examples indicate that in general, a subadditive operator need not be countably subadditive. For instance, we may define $T$ from $L^1([0,1])$ to $\rrr$  via
$$
T(f)= \min\Big\{1, \limsup_{n\to \infty} \, n\!\int_0^{\frac1n} |f(x)|\,dx \Big\}.
$$
The  functional $T$ is certainly subadditive but not countably subadditive on $L^1([0,1])$, as the sequence $ f_j= \chi_{(\frac1{j+1}, \frac1j]}$, $j \in \zzz^+$, violates \eqref{1}.

A positive answer to the question of when subadditivity implies countable subadditivity can be given if the (not necessarily linear) operator $T$ is assumed to be {\it ``continuous at zero''} in the sense that given $\ve>0$ there exists $\de>0$ such that 
\bee\lab{2}
f\in W, \q | f|_W <\de\implies | T(f) |_V <\ve.  
\eee
This property implies that the action of $T$ on the tail of a convergent series tends to zero in norm, and countable subadditivity may be easily deduced from this and the subadditivity property. 

Although continuity implies countable subadditivity, it should be noted that the reverse implication is not valid. For instance, the linear operator $L(f) (x) = g(x) \int_{\rrr} f(t)\, dt$ 
is countably subadditive on $L^1(\rrr)$, but it is not continuous from $L^1(\rrr)$ to any reasonable space, if $g$ is a measurable function that exhibits bad behavior everywhere. 

In many situations arising in analysis, $V$ and $W$ are not normed spaces but rather quasi-normed spaces, which means that the triangle inequality in \eqref{0} holds with the appearance of a multiplicative constant on the right-hand side. Quasi-normed spaces are ubiquitous in analysis: for instance, Hardy spaces, $L^p$ spaces, and Lorentz spaces with indices less than $1$ are all examples of such spaces.

When $V$ and $W$ are quasi-normed spaces, it is reasonable to replace \eqref{0} by the {\it quasi-additivity} condition
\bee\lab{3}
\big| T(f+g) \big|_{V} \le K\big( \big| T(f ) \big|_{V}+\big| T( g) \big|_{V}\big) 
\eee
for all $f,g $ in $W$ and some fixed constant $K> 1$.  Then under the assumptions \eqref{2} and \eqref{3}, $T$ enjoys the {\it countable $\alpha$-quasi-additivity} property 
\bee\lab{5}
\Big| T\Big(\sum_{j\in \zzz} f_j \Big) \Big|_{V}^\al \le 4\sum_{j\in \zzz}\big| T(f_j ) \big|_{V}^\al 
\eee
for all series $\sum_{j\in \zzz} f_j$ that converge in the quasi-norm of $W$. Here $\al$  is the positive constant that satisfies  $(2K)^\al=2$.

Assertion \eqref{5} can be derived as follows: by the Aoki--Rolewicz theorem~\cite{A,KPR,R} one has   
\bee\lab{6}
\Big| T\Big(\sum_{j\in \zzz} f_j \Big) \Big|_{V}^\al \le 4\sum_{|j|\le N}\big| T(f_j ) \big|_{V}^\al 
+ 4 \,\Big| T\Big(\sum_{|j|>N} f_j \Big) \Big|_{V} ^\al 
\eee
for any $N\in \zp$.  Then given $\ve>0$, the last term on the right-hand side in \eqref{6} can be made smaller than $4\, \ve^\al$, provided $N$ is large enough such that $| \sum_{|j|>N} f_j|_W<\de$. Letting 
$N\to \nf$ we obtain 
\bee\lab{7}
\Big| T\Big(\sum_{j\in \zzz} f_j \Big) \Big|_{V}^\al \le 4\sum_{j\in \zzz}\big| T(f_j ) \big|_{V}^\al 
+ 4\, \ve^\al  
\eee
and, as $\ve>0$ was arbitrary, we deduce \eqref{5}. 

In this note, we consider the issue of countable subadditivity and countable $\alpha$-quasi-additivity in the situation when the underlying spaces $W$ and $V$ are Lorentz spaces over $\si$-finite measure spaces.  Our main observation is that the continuity assertion \eqref{2} can be derived from weak-type, or even restricted weak-type estimates.  This is encapsulated in Proposition~\ref{P1}, which guarantees that a subadditive operator satisfying the hypotheses of Hunt's interpolation theorem is countably subadditive on the interpolated spaces.  As observed earlier, this allows one to prove the Hunt theorem directly (without resorting to simple functions) and uniformly on all Lorentz spaces, even those in which simple functions are not dense.

Further applications are discussed in Section \ref{S:Applications}; these include extensions of the Yano extrapolation theorem and the Calder\'on--Zygmund theorem, as well as a result concerning $0$-local operators.

\subsection{Acknowledgements} 
L.G. is supported by the Simons Foundation Grant 624733. 
M.V. was supported by NSF grant DMS--2054194.

\section{Countable subadditivity and $\alpha$-quasi-additivitity for operators between Lorentz spaces}

In this section, we show that mild boundedness assumptions   imply countable subadditivity and countable $\alpha$-quasi-additivitity.

To fix notation, we let $(X,\mu)$ and $(Y,\nu)$ be two $\si$-finite measure spaces.  We denote by $ S(X)$ the space of complex-valued simple functions on $X$ and by $\mathscr M(X)$ the space of complex-valued measurable functions on $X$.  We analogously define $S(Y)$ and $\mathscr M(Y)$.

Following \cite{CFA}, we denote by $S_0 (X)$ the subset of $S(X)$ of functions of the form $f_1-f_2+if_3-if_4$, where each $f_j$ has the form 
$
\sum_{k= m }^{n } 2^{-k} \chi_{A_k}
$
where $m <n $ are integers and $A_k$ are subsets of $X$ of finite measure.  
  
\begin{dfn} \label{D:Lorentz}
The {\it decreasing rearrangement} of a function $f\in \mathscr M(X)$ is defined as follows: for any $t\in[0,\infty)$,
\begin{equation}
f^*(t)=\inf \{s>0:\,\, \mu(\{ |f|> s\}) \le t\} .
\end{equation}    
For  $0<p,q\le \infty$, the {\it Lorentz space} $L^{p,q}(X)$ is the space of all complex-valued measurable functions $f$ on $X$ for which the quasi-norm 
\begin{equation*} 
\big\|f\big\|_{L^{p,q}(X)} =\begin{cases}
\left(\displaystyle\int_0^\infty\left(t^\frac{1}{p} f^*(t)\right)^q\,
\dfrac{dt}{t}\right)^\frac{1}{q}
\,\,&\,\,\text{if $q<\infty\!$ ,} \\
\sup\limits_{t>0} \,t^\frac{1}{p} f^*(t)\quad\quad
\,\,&\,\,\text{if $q=\infty\!$ } \end{cases}
\end{equation*}
is finite.
\end{dfn}

Note that $L^{p,p}(X)=L^p(X)$ for all $0<p\leq \infty$, while $L^{p,\infty}(X)$ coincides with the weak $L^p(X)$ space for $0<p<\infty$.

Lorentz spaces satisfy a nesting property; specifically, $L^{p,q}(X)\subset L^{p,r}(X)$ whenever $0<p\leq \infty$ and $0<q<r\leq \infty$.  Indeed, there exists a constant $C(p,q,r)$ depending only on $p,q,r$ so that
\begin{align}\label{nesting}
\big\|f\big\|_{L^{p,q}(X)}\leq C(p,q,r)\big\|f\big\|_{L^{p,r}(X)}.
\end{align}

That the space of simple functions $S(X)$ is dense in $L^{p,q}(X)$ for all $0<p,q<\infty$ was observed already in \cite{Hunt66}.  For a proof that $S_0(X)$ is dense in $L^{p,q}(X)$ for $0<p,q<\infty$ see Proposition 1.4.21 in \cite{CFA}. However, $S(X)$ is \emph{not} dense in $L^{p,\infty}(X)$ for any $0<p\leq\infty$ if $\mu$ has infinite support; see Exercise 1.4.4. in \cite{CFA}.

\begin{dfn}\label{D1}
Let $T$ be a mapping defined on a Lorentz space $L^{p,q}(X)$ and taking values in the space $\mathscr M(Y)$. We say that 
\begin{enumerate}

\item[(a)] $T$ is \emph{subadditive} on $L^{p,q}(X) $ if for all $f,g\in L^{p,q}(X)$ we have 
\[
|T(f+g)|\le |T(f )|+|T( g)| \qq \textup{$\nu$-a.e.}
\]
If in addition $T$ satisfies
\begin{align}\label{homogeneous}
|T(\lambda f)| =|\lambda| |T(f)| \quad\text{for all $\lambda\in \ccc$ and $f\in L^{p,q}(X)$,}
\end{align}
then $T$ is called \emph{sublinear}.  

\item[(b)] $T$ is \emph{countably subadditive} on $L^{p,q}(X)$ if whenever $\sum_{j\in \zzz} f_j$ converges in $L^{p,q}(X)$, we have
\[
\Big|T\Big(\sum_{j\in \zzz} f_j\Big)\Big|\le \sum_{j\in \zzz} |T(f_j )|  \qq \textup{$\nu$-a.e.}
\]

\item[(c)] $T$ is \emph{quasi-additive} on $L^{p,q}(X)$ if there is a constant $K>1$ such that  
for all $f,g\in L^{p,q}(X)$ we have 
\[
|T(f+g)|\le K( |T(f )|+|T( g)|)\qq \textup{$\nu$-a.e.}
\]
If in addition $T$ satisfies \eqref{homogeneous}, then $T$ is called \emph{quasi-linear}.  

\item[(d)] $T$ is \emph{countably $\al$-quasi-additive} on $L^{p,q}(X)$ for some $0<\al \le 1$ if there is a constant $K' \ge 1$ such that whenever $\sum_{j\in \zzz} f_j$ converges in $L^{p,q}(X)$, we have
\[
\Big|T\Big(\sum_{j\in \zzz} f_j\Big)\Big|^\al \le K' \sum_{j\in \zzz} |T(f_j )|^\al \qq \textup{$\nu$-a.e.}
\]
\end{enumerate}
\end{dfn}

Next, we recall the restricted weak-type and weak-type conditions.

\begin{dfn}\label{D2}
We say that a mapping $T$ defined on a subset of $\mathscr M(X)$ containing $S(X)$ and taking values in $\mathscr M(Y )$ is of \emph{restricted weak-type} $(p,q)$ for $0<p<\infty$ and $0<q\le \nf$ if there is a constant $C>0$ so~that 
\begin{equation}\lab{1.4.RWT66}
\big\|T(\chi_A)\big\|_{\wlq(Y)} \le C\, \mu(A)^{1/p}
\end{equation}
for all measurable subsets $A$ of $X$  with finite measure.

We say that $T$ is of \emph{weak-type} $(p,q)$ if there is a constant $C>0$ such that 
\begin{equation}\lab{1.4.RWT66}
\big\|T(f)\big\|_{\wlq(Y)} \le C\, \|f\|_{L^p(X)}
\end{equation}
for all $f\in L^p(X)$.
\end{dfn}

We recall that a sublinear operator $T$ is of restricted weak-type $(p,q)$ with $0<p<\infty$ and $1<q<\infty$ if and only if it admits a bounded extension from $L^{p,1}(X)$ to $L^{q, \infty}(Y)$; see, for example, \cite[Exercise 1.4.7]{CFA}.

As observed in the introduction, countable subadditivity and countable $\alpha$-quasi-additivity are consequences of continuity.  We first observe that the weak-type condition provides a strong form of continuity. 

\begin{prop}\label{P00} Fix $0<p<\infty$ and $0<q\le \nf$.\\[1mm]
(i) A subadditive operator $T$ of weak-type (p,q) is countably subadditive on $L^p(X)$. \\[1mm]
(ii) A quasi-additive operator $T$ of weak-type (p,q) is countably $\al$-quasi-additive on $L^p(X)$.
\end{prop}

\begin{proof}
We prove (ii) as the proof of (i) is easier. 

Let  $\sum_{j\in \zzz} f_j$ be a series converging in $L^p(X)$. By the Aoki--Rolewicz theorem (with $(2K)^\al =2$), we obtain 
\bee\lab{8}
\Big| T\Big(\sum_{j\in \zzz} f_j \Big) \Big| ^\al \le 4\sum_{|j|\le N}\big| T(f_j ) \big| ^\al 
+ 4\, \Big| T\Big(\sum_{|j|>N} f_j \Big) \Big| ^\al \q\textup{$\nu$-a.e.} 
\eee
for any $N\in \zp$.

To conclude the proof, it thus suffices to show that $T\big(\sum_{|j|>N} f_j \big)$ converges to zero $\nu$-a.e. as $N\to \nf$. As $T$ is of weak-type $(p,q)$, we have 
$$
\Bigl\|T\big(\sum_{|j|>N} f_j \big) \Bigr\|_{\wlq(Y)} \leq C \Bigl\| \sum_{|j|>N} f_j \Bigr\|_{L^p(X)}
$$
for some fixed $C>0$.  As the right-hand side converges to zero as $N\to\nf$, we conclude that $T\big(\sum_{|j|>N} f_j \big)$ converges in measure to zero.  Consequently, there is a subsequence $N_k\to \infty$ such that $T\big(\sum_{|j|>N_k} f_j \big)\to 0$ $\nu$-a.e. Setting $N=N_k$ in \eqref{8} and letting $k\to \nf$ we obtain that $T$ is countably $\al$-quasi-additive on $L^p(X)$.
\end{proof}


Next we discuss a weaker form of continuity that is still sufficient to imply countable subadditivity in the setting of Lorentz spaces. 
 
\begin{prop}\lab{P0}
Let $0<p < \nf$, $0<q\le \nf$, and $0<r<\min \{1,q\}$. Let $T$ be a subadditive operator defined on $S_0(X)$ with values in  $\mathscr M(Y)$ which is of restricted weak-type 
$(p,q)$. Then $|T|$ has a unique countably subadditive extension on $L^{p, r}(X)$.  
\end{prop}

\begin{proof}
By Lemma 1.4.20 in \cite{CFA} there is a constant $C>0$ such that 
\bee\lab{9}
\| T(\vp)\|_{L^{q,\nf} (Y)} \le C \| \vp\|_{L^{p, r} (X)}
\eee
for all functions $\vp$ in the subspace $S_0(X)$ of $L^{p, r} (X)$. 

Next, we use the density of $S_0(X)$ in $L^{p, r} (X)$ together with \eqref{9} to define an extension of $|T|$ on $L^{p, r} (X)$.  Given a  sequence $\{\vp_j\}_j\subset S_0(X)$ converging to a given $f\in L^{p, r} (X)$ in the norm of this space, we claim that $|T(\vp_j)|$ converges to $|T(f)|$ in $L^{q,\nf} (Y)$. To see this, we note that in view of subadditivity we have 
\[
\big| |T(\vp_{j'} )| - |T(\vp_j)| \big| \le |T(\vp_{j'}-\vp_j)|.
\]
Invoking \eqref{9}, this yields that that sequence $\{|T(\vp_j)|\}_j$ is Cauchy in the complete metric space $L^{q,\nf}(Y)$ and thus it has a limit in this space, which we denote by $|T(f)|^*$.
Moreover,
\begin{align}\label{bdd}
\| |T(f)|^*\|_{L^{q,\nf} (Y)} &=\lim_{j\to \infty} \| T(\vp_j)\|_{L^{q,\nf} (Y)}\notag\\
&\leq C \lim_{j\to \infty}\| \vp_j\|_{L^{p, r} (X)} = C \| f\|_{L^{p, r} (X)}.
\end{align}

Next we observe that the object $|T(f)|^*$ does not depend on the particular sequence $\{\vp_j\}_j$ used to approximate $f$. Indeed, if $\{\psi_j\}_j$ is another sequence in $S_0(X)$ 
converging to $f$, then 
\[
\big| |T(\psi_{j} )| - |T(\vp_j)| \big| \le |T(\psi_{j}-\vp_j)|.
\]
Together with \eqref{9} this shows that $|T(\psi_{j} )|$ has the same limit as $|T(\vp_{j} )|$ in $L^{q,\infty}(Y)$.  Therefore, $|T(f)|^*$ is well defined for each $f\in L^{p,r}(X)$. 

This allows us to define an extension $|T|^*$ of $|T|$  on $L^{p, r}(X)$ by setting $|T|^*(f) = |T(f)|^*$.  Moreover, in view of \eqref{bdd}, this extension satisfies
\begin{align}\label{bdd'}
\| |T(f)|^*\|_{L^{q,\nf} (Y)}\leq  C \| f\|_{L^{p, r} (X)}.
\end{align}

This bound provides the required continuity needed to show that the extension $|T|^*$ is a countably subadditive operator on $L^{p, r}(X)$.  Indeed, arguing as in the proof of Proposition~\ref{P00}, to derive countable subadditivity we need to control  $| T (\sum_{|j|>N} f_j)|^*$ whenever $\sum_{|j|>N} f_j\to 0 $ in $L^{p, r} (X)$. The estimate \eqref{bdd'} implies that $\| |T (\sum_{|j|>N} f_j)|^*\|_{L^{q,\nf}(Y)}$ tends to zero as $N\to \nf$, so there is a subsequence $N_k\to \infty$ such that $|T (\sum_{|j|>N_k} f_j)|^*$ tends to zero $\nu$-a.e. as $k\to \infty$. This yields the countable subadditivity of the extension $|T|^*$.
\end{proof} 
 
\begin{rmk}
Proposition~\ref{P0} has a natural extension in the case $T$ is a quasi-additive operator defined on $S_0(X)$ with values in $\mathscr M(Y)$ which is of restricted weak-type $(p,q)$ for $0<p < \nf$ and $0<q\le \nf$.  In this setting, Lemma 1.4.20 in \cite{CFA} establishes the estimate \eqref{9} provided
$$
0<r<\min \bigl\{q, \tfrac{\log 2}{\log(2K)}\bigr\},
$$
where $K$ denotes the quasi-additivity constant of $T$.  The arguments in Proposition~\ref{P0} then show that $|T|$ has a unique countably $\alpha$-quasi-additive extension on $L^{p, r}(X)$ where $\alpha$ is chosen such that $(2K)^\alpha=2$.
\end{rmk}

Proposition~\ref{P0} shows that operators that are well-behaved on simple functions have extensions that are likewise well-behaved. However, the extension provided by Proposition~\ref{P0} need not be the only extension of the operator, as the following example demonstrates: Let $B$ denote a vector space basis of $S_0(X)$ and let $B'$ denote an extension of $B$ to a basis of $L^{p,r}(X)$.  Given a function $h\neq 0$, that may or may not belong to $L^{q, \infty}(Y)$, we define
\begin{equation*}
T(f) =\begin{cases} 0, &\quad\text{if $f\in B$}\\ h, &\quad\text{if $f\in B'\setminus B$},\end{cases}
\end{equation*}
and extend $T$ to $S_0(X)$ and $L^{p,r}(X)$ as a linear operator. Then the extension provided by Proposition~\ref{P0} is the zero operator on $L^{p,r} (X)$ and so does not coincide with $T$ on $L^{p,r}(X)$.


In fact, the restricted weak-type property is most useful when it is additionally known that the operator is also bounded on some other space, in which case there is no ambiguity over the extension.

Our main result in this section demonstrates that if two estimates are known (such as in the setting of the Hunt or Marcinkiewicz interpolation theorems, for example), then subadditivity (respectively, quasi-additivity) properties of the operator necessarily imply countable subadditivity (respectively, countable $\al$-quasi-additivity) on interpolation spaces.

\begin{prop}\label{P1}
Let $0<p_0 <p<  p_1 \le \nf$ and $0<q_0, q_1\le \nf$. Let $T$ be a quasi-additive (respectively, subadditive) operator defined on $L^{p_0,1}(X)+L^{p_1,1}(X)$ with values in $\mathscr M(Y)$. Assume that $T$ satisfies
\begin{align}
\|T(f)\|_{L^{q_1,\infty}(Y)} \leq C_0 \|f\|_{L^{p_1,1}(X)} \label{hyp1},\\
\|T(f)\|_{L^{q_2,\infty}(Y)} \leq C_1 \|f\|_{L^{p_2,1}(X)}\label{hyp2},
\end{align}
for fixed constants positive $C_0,C_1$. Then $T$ is countably $\al$-quasi-additive (respectively, countably subadditive) on $L^{p,r}(X)$ for any $0< r\le \infty $, where $\al$ is related to the quasi-addivity constant $K$ via $(2K)^\al=2$. 
\end{prop}

\begin{proof}
We prove the claim in the case $T$ is quasi-additive, as the subadditive case is included.

Given $f\in L^{p,r}(X)$, we may decompose it as 
$$
f=f_0+f_1 \quad\text{with  $f_0=f\chi_{\{|f|>f^*(1)\}}$.}
$$
Then
\begin{equation*}
f_0^*(t)\leq f^*(t) \chi_{(0,1)}(t) \qquad\text{and} \qquad f_1^*(t)\leq \begin{cases} f^*(1),\quad \text{if $t<1$}\\  f^*(t), \quad \text{if $t\geq 1$}\end{cases}
\end{equation*}
which yields
 \begin{equation}\label{10}
 \| f_0\|_{L^{p_0,1}}+  \| f_1\|_{L^{p_1,1}} \le C(p,p_0,p_1) \|f\|_{L^{p,\nf}} \le  C(p,p_0,p_1) \|f\|_{L^{p,r}},
 \end{equation}
The last inequality above follows from the nesting property \eqref{nesting} of Lorentz spaces.

As \eqref{10} shows, the space $L^{p,r}(X)$ is contained in $L^{p_0,1}(X)+L^{p_1,1}(X)$.  Therefore, given a sequence $f_j$ in $L^{p,r}(X)$, we may decompose each $f_j=g_j+h_j$ with $g_j\in L^{p_0,1}(X)$ and $h_j \in L^{p_1,1}(X)$.  If the series $\sum_j f_j$ converges in $L^{p,r}(X)$,  then \eqref{10} shows that $\sum_j g_j$ converges in $L^{p_0,1}(X)$ and $\sum_j h_j$ converges in $L^{p_1,1}(X)$.  

The quasi-additivity of $T$ yields 
 \begin{equation}\label{11}
 \Big| T\big( \sum_{|j|>N} f_j\big) \Big| \le K 
  \Big| T\big( \sum_{|j|>N} g_j\big) \Big| + K \Big| T\big( \sum_{|j|>N} h_j\big) \Big|.
 \end{equation}
 By \eqref{hyp1}, \eqref{hyp2}, and the convergence of $\sum_j g_j$ in $L^{p_0,1}(X)$ and that of $\sum_j h_j$ in $L^{p_1,1}(X)$, we have that $T\big( \sum_{|j|>N} g_j\big)$ tends to zero in $L^{q_0,\nf}(Y) $ and $T\big( \sum_{|j|>N} h_j\big)$ tends to zero in $L^{q_1,\nf}(Y)$. Passing to a subsequence $N_k\to \infty$, we may thus guarantee that $T\big( \sum_{|j|>N_k} g_j\big)$ tends to zero $\nu$-a.e. and $T\big( \sum_{|j|>N_{k}} h_j\big)$ tends to zero $\nu$-a.e. as $k\to \nf$.  By \eqref{11}, this shows that $T\big( \sum_{|j|>N_{k}} f_j\big)$ tends to zero $\nu$-a.e. as $k\to \nf$.
 
To continue, we employ the Aoki--Rolewicz theorem to deduce 
 \begin{equation}\label{12}
 \Big| T\big( \sum_{|j|\le N_k} f_j\big) \Big|^\al \le 4  \sum_{|j|\le N_k} \big| T ( f_j ) \big|^\al  + 4  \Big| T\big( \sum_{|j|> N_k} f_j\big) \Big|^\al 
 \q \textup{$\nu $-a.e.}
 \end{equation}
Sending $\kappa\to\infty$, this yields the countable $\al$-quasi-additivity of $T$.
\end{proof}

Note that the boundedness properties \eqref{hyp1} and \eqref{hyp2} of $T$ are the standard hypotheses in Hunt's interpolation theorem \cite{Hunt64, Hunt66}.  Moreover, by the nesting property \eqref{nesting}, if $T$ is of weak-type $(p_0,q_0)$ and of weak-type $(p_1,q_1)$, then it also satisfies \eqref{hyp1} and \eqref{hyp2}; weak-type bounds are the standard hypotheses in the Marcinkiewicz interpolation theorem.

Proposition~\ref{P1} shows that the colloquial formulation of the Hunt interpolation theorem holds true --- there is no need to discuss simple functions.  Indeed, Proposition~\ref{P1} guarantees that, under the traditional hypotheses of the Hunt interpolation theorem, the operator $T$ extends naturally to the entirety of $L^{p,r}(X)$, even when $r=\infty$. By comparison, traditional treatments guarantee only that the interpolation bounds hold for simple functions and consequently, that the operator $T$ extends to a bounded operator on $L^{p,r}$ only when simple functions are dense therein.  As discussed in the introduction, this excludes the weak $L^p$ spaces.

\section{Applications}\label{S:Applications}

In this section, we look at three examples where countable subadditivity plays a critical role. 

\subsection{Yano's extrapolation}
For our first application we prove a slight extension of Yano's extrapolation theorem \cite{Y}. 

\begin{thm}
Let $(X,\mu )$ and $(Y,\nu)$ be finite measure spaces and fix $1<p_*<\nf$ and positive constants $A, \al$.  Let $T$ be a sublinear operator that maps $L^{p,1}(X)$ to $L^{p, \infty}(Y)$ for every $1<p < p_* $ with norm bounded by $A (p-1)^{-\al}$. Then for all $f\in \cup_{1<p< p_*} L^{p,1}(X)$ we have
\begin{equation}\label{yano}
\int_Y |T(f)|\,d\nu \le   A \, C_Y \bigg[\int_X |f|(\log_2^+ |f|)^\al \,d\mu +C_{X,\al, p_*}\bigg]\, ,
\end{equation}
where $C_{X,\al, p_*}$ and $C_Y$ are constants depending on the indicated parameters. 
\end{thm}

\begin{proof} 
By Proposition~\ref{P1}, $T$ is countably subadditive on $L^{p,1}(X)$ for all $1<p<p_*$.

Given $f\in L^{p,1}(X)$ for some $1<p< p_*$, we decompose
\[
f=\sum_{k=0}^\nf f\chi_{S_k} \, ,
\]
where $S_0=\{|f|<2\}$ and $S_k=\{ 2^k \le |f|<2^{k+1}\}$ for $k\ge 1$.  That the series representation of $f$ converges in $L^{p,1}(X)$ follows from the dominated convergence theorem.

Let $k_0\geq 1$ be such that $ 1+ \frac{1}k< p_*$ for all $k\geq k_0$.  Let $p_k= \frac{k+1}{k}$ whenever $k\geq k_0$.  

By the countable subadditivity property of $T$ and H\"older's inequality in Lorentz spaces, we may estimate
\begin{align*}
&\int_Y |T(f)|\,d\nu\\
&\le \sum_{k\geq 0} \int_Y |T(f\chi_{S_k})|\,d\nu \\
&\le \!\!\sum_{0\leq k< k_0} \!\!\|T(f\chi_{S_k})\|_{L^{p,\infty}(Y)}\|1\|_{L^{p',1}(Y)}+\sum_{k\geq k_0}  \|T(f\chi_{S_k})\|_{L^{p_k, \infty}(Y)}\|1\|_{L^{p_k',1}(Y)} \\
&\le \!\frac A{ (p-1)^{\alpha}} \!\!\sum_{0\leq k< k_0}\!\!\! \|f\chi_{S_k}\|_{L^{p,1}(X)} \nu(Y)^{\frac1{p'}}+\!\sum_{k\geq k_0} \!\frac A{ (p_k-1)^{\alpha}}\|f\chi_{S_k}\|_{L^{p_k,1}(X)} \nu(Y)^{\frac1{p_k'}} \\
&\le A C_{Y}\Bigl[ C_{X, \alpha, p_*} + \sum_{k\geq k_0} 2^{k+1} k^\alpha \mu(S_k)^{\frac{k}{k+1}}\Bigr].
\end{align*}
Claim \eqref{yano} follows from this and an application of Young's inequality:
$$
\mu(S_k)^{\frac{k}{k+1}} \leq \tfrac{k}{k+1} \, 4\mu(S_k) + \tfrac{1}{k+1} \, 4^{-k}\leq 4\mu(S_k) +  4^{-k}.
$$
Summing over $k\ge k_0$ yields the claimed assertion.
\end{proof}

\bigskip

\subsection{The Calder\'on--Zygmund theorem}

Our next application concerns singular integral operators of Calder\'on--Zygmund type.  One of the central themes in the treatment of such operators is that knowledge of the boundedness of the operator from $L^r(\rn)$ to $L^r(\rn)$ for some $1<r<\infty$, together with further regularity conditions satisfied by the kernel, guarantee that the operator is of weak-type $(1,1)$.  An application of the Marcinkiewicz interpolation theorem then allows one to conclude that the operator admits a bounded extension from $L^p(\rn)$ to $L^p(\rn)$ for all $1<p\leq r$; see \cite{CFA, Stein}.  Here we   discuss an analogous statement under the milder assumption that $T$ is bounded from $L^{r,1}(\rn)$ to $L^{r,\infty}(\rn)$ for some fixed $1<r<\infty$.

\begin{thm}
Suppose that the singular integral operator $T$ is associated with a kernel $K(x,y)$ in the sense that
\begin{align}\label{rep}
T(f)(x) = \int_{\rn} K(x,y) f(y) \, dy \qqq \text{for $x\notin \textup{supp } f$}, 
\end{align}
whenever $f$ is compactly supported.  Assume that there exists two constants $A,A'>0$ and $1<r<\infty$ so that\\[1mm]
\noindent (i) $|K(x,y)|\le A \, |x-y|^{-n}$ uniformly for $x\neq y$,\\[1mm]
\noindent(ii) $\int_{|x-y|\ge 2|y-y'|} |K(x,y)-K(x,y')|\, dx \le A'$ uniformly for $y, y'\in \R^d$,\\[1mm]
\noindent (iii) $T$ is bounded from $L^{r,1}(\rn)$ to $L^{r,\nf}(\rn)$.

Then $T$ admits an extension that is of weak-type $(1,1)$ with a bound proportional to 
$A'+\|T\|_{L^{r,1} \to L^{r,\nf}}$.
\end{thm}

\begin{proof} Fix $f\in L^1\cap L^{r,1}$, which is a dense subset of $L^1$.  For $\lambda>0$ fixed and a constant $\gamma>0$ to be specified shortly, we perform a Calder\'on--Zygmund decomposition of $f$ at level $\ga\lambda$, thereby writing $f=g+\sum_{j\in \zzz} b_j$ where 
\begin{align}\label{g}
\|g\|_{L^1} \leq \|f\|_{L^1} \qq\text{and} \qq |g|\leq 2^n \ga\lambda \q \text{a.e.}
\end{align}
and each $b_j$ is supported in a dyadic cube $Q_j$ such that
\begin{align}\label{b}
\|b_j\|_{L^1} \leq 2^{n+1} \ga\lambda |Q_j|  \qq\text{and} \qq \int_{Q_j} b_j(y)\, dy = 0.
\end{align}
Moreover, the dyadic cubes $Q_j$ have pairwise disjoint interiors and satisfy
\begin{align}\label{cubes}
\sum_{j\in \zzz}|Q_j|\leq (\ga\lambda)^{-1} \|f\|_{L_1}.
\end{align}

For each cube $Q_j$, let $y_j$ denote its center.  Let $Q_j^*$ denote the dilate of $Q_j$ with the same center $y_j$ and side $\ell(Q_j^*)=2\sqrt{n}\, \ell(Q_j) $.  Note that in view of the hypothesis (i), $T(b_j)$ admits the representation \eqref{rep} for $x\notin Q_j^*$.  Moreover, as
$$
|x-y_j|\ge 2|y-y_j| \qq\text{for all $y\in Q_j$ and $x\notin Q_j^*$,}
$$
hypotheses $(i)$ and $(ii)$ together with \eqref{b} yield
\begin{align}\label{b'}
\int_{\rn\setminus Q_j^*} |T(b_j)(x)| \, dx &\leq  \int_{\rn\setminus Q_j^*}\int_{Q_j} |K(x,y)-K(x,y_j)| |b_j(y)|\,dy\, dx \notag \\
&\leq  2^{n+1} A' \gamma\lambda |Q_j|.
\end{align}

By \eqref{g} and interpolation, we conclude that $g\in L^{r,1}$.  Indeed, 
\begin{equation}\label{bbb}
\|g\|_{L^{r,1}} \lesssim \|f \|_{L^1}^{\f 1r} (\ga \la)^{1-\f 1r}.
\end{equation}
Consequently, the series $\sum_{j\in \zzz} b_j= f-g \in L^{r,1}$ and converges a.e.  By the dominated convergence theorem, it follows that the series converges in $L^{r,1}$.  Invoking hypotheses (iii), we deduce that $|T (\sum_{|j|>N} b_j) |$ converges to zero in $L^{r,\nf}$ as $N\to \nf$.  A familiar argument by now then yields that
\begin{equation}\label{fix}
T(f) = T(g) + \sum_j  T(b_j)  \q \text{a.e.}
\end{equation}

To continue, we use hypothesis (iii), \eqref{g}, \eqref{cubes}, \eqref{b'}, \eqref{bbb}, 
and \eqref{fix} to 
derive the bound
\begin{align*}
&\Big| \{ x:\, |T(f)(x)|>\lambda \} \Big|\\
&\leq \Big| \{ x:\, |T(g)(x)|>\tfrac\lambda2 \} \Big|  + \sum_{j\in \zzz} |Q_j^*| + \Big| \{ x\notin \cup Q_j^*:\, \sum_{j\in\zzz}|T(b_j)(x)|>\tfrac\lambda2 \} \Big|\\
&\lesssim \big[ \lambda^{-1} \|T(g)\|_{L^{r, \infty}}\big]^r + (\ga\lambda)^{-1}\|f\|_{L^1} +  \lambda^{-1}\int_{\rn\setminus \cup Q_j^*}  \sum_{j\in\zzz}|T(b_j)(x)|\, dx\\
&\lesssim \big[\lambda^{-1} \| T\|_{L^{r,1}\to L^{r,\nf}} 
\|f \|_{L^1}^{\f 1r} (\ga \la)^{1-\f 1r}\big]^r+ (\ga\lambda)^{-1}\|f\|_{L^1}+A' \lambda^{-1}\|f\|_{L^1}\\
&\lesssim \lambda^{-1}(A'+\|T\|_{L^{r,1} \to L^{r,\nf}}) \|f\|_{L^1},
\end{align*}
having chosen $\ga= 1/\|T\|_{L^{r,1} \to L^{r,\nf}}$. All
  the implicit constants depend only on the dimension (and not on $A$, $f$, $\la$). 

This proves that $T$ satisfies weak-type $(1,1)$ bounds on $f\in L^1\cap L^{r,1}$.  Using the density of $L^1\cap L^{r,1}$ in $L^1$, we may therefore extend $T$ to a bounded operator from $L^1(\rn)$ to $L^{1,\infty}(\rn)$.
\end{proof}

\subsection{Subadditive $0$-local operators}

Finally,  we consider an application concerning subadditive operators that preserve the supports of functions with vanishing integral; we call such operators {\it $0$-local}.

\begin{thm}\label{local}
Let $1<r\le \infty$. Suppose $S$ is a subadditive operator defined on $L^1(\rn)+L^{r,1}(\rn)$ that is bounded from $L^{r,1}(\rn)$ to $L^{r,\infty}(\rn)$.
Assume that $S$ is  $0$-local in the sense that whenever $h$ is supported in a dyadic cube $Q$ and has vanishing integral, then 
$S(h)$ is supported in a fixed multiple $Q^*$ of $Q$.   Then  $S$ is of weak-type $(1,1)$. 
\end{thm}

\begin{proof}  Assume initially that $f $ lies in $ L^1 \cap L^{r,1} $, which is dense in $ L^1$.  For $\lambda>0$ fixed and a constant $\gamma>0$ to be chosen shortly, we apply the Calder\'on--Zygmund  decomposition to $f$ at height $\ga\la$.  In this way, we may write $ f=g+b$  so that \eqref{g} and \eqref{b} are satisfied.  As $g\in L^1\cap L^\infty$, we have $g\in L^{r,1}$ and it satisfies the bound \eqref{bbb}.  By hypothesis, this implies
\begin{equation}\label{sg}
\|S(g)\|_{L^{r,\nf}}\lesssim \| S\|_{L^{r,1}\to L^{r,\nf}}  \|f \|_{L^1}^{\f 1r} (\ga \la)^{1-\f 1r}.
\end{equation}

As $b=f-g\in L^{r,1}$ and $S$ is bounded from $L^{r,1}$ to $L^{r,\infty}$, we may deduce the countable subadditivity property
$$
| S(b) | \le \sum_j |S(b_j) |   \qq\textup{a.e.}
$$
As each $b_j$ is supported in $Q_j $ and has mean value zero, our hypotheses guarantee that $S(b_j)$ is supported in $Q_j^*$.  Hence, $\sum_j |S(b_j) |$  is supported in $\cup_j Q_j^*$, which has  measure bounded by a constant multiple of $(\ga\lambda)^{-1}\|f\|_{L^1}$. 

Using that $|S(f)|\le |S(g)|+|S(b)|$ together with \eqref{sg} and choosing  $\ga=   1/\|S\|_{L^{r,1} \to L^{r,\nf}}$, we may bound
\begin{align*}
\lambda \bigl| \bigl\{ |S(f)|>\lambda\bigr\}\bigr| &\leq \lambda\bigl| \bigl\{  |S(g)|>\tfrac\lambda2\bigr\}\bigr| +\lambda\bigl| \bigl\{  |S(b)|>\tfrac\lambda2\bigr\}\bigr|\\
&\lesssim \lambda \bigl(\tfrac{ \|S(g)\|_{L^{r,\infty}}}{\lambda}\bigr)^r + \lambda(\ga\lambda)^{-1}\|f\|_{L^1}\\
&\lesssim \|S\|_{L^{r,1} \to L^{r,\nf}} \|f\|_{L^1},
\end{align*}
where all implicit constants are independent of $\lambda$.  This proves the weak-type $(1,1)$ bounds for functions $f\in L^1\cap L^{r,1}$.  Since  
$S$ is already defined on $L^1$, a density argument yields
$$
\|S(f)\|_{L^{1,\infty}}\lesssim \|S\|_{L^{r,1} \to L^{r,\nf}} \|f\|_{L^1}
$$
for all $f\in L^1$. 
\end{proof}

Note that by an application of Hunt's interpolation theorem, the operator $S$ in Theorem~\ref{local} maps $L^p(\rn)$  to $L^p(\rn) $ for every $1<p<r$.
 
If the operator $S$ were only defined on $L^{r,1}$ and not on $L^1$, then we can deduce that it has a unique bounded subadditive extension on $L^1$ which is of weak-type $(1,1)$ and a unique bounded subadditive extension on $L^p$ for $1<p<r$.  

Examples of  $0$-local  subadditive operators in the sense of Theorem~\ref{local}  
can be constructed as follows: Let $h_I$ be the Haar functions, where $I$ ranges over all 
dyadic intervals. These are equal to $| I |^{-1/2}$ on the left half of $I$ and $-| I |^{-1/2}$ on the right half.
Let $\mathcal  D_k$ be the set of all 
dyadic intervals of length $2^{-k}$. The dyadic martingale difference operator is 
\bee\label{sum2}
D_k (f) = \sum_{I\in\mathcal  D_k} \langle f, h_I\rangle h_I, \qq f\in L^1_{loc}(\mathbf R).
\eee
For a bounded and compactly supported sequence $\{a_{k,j}\}_{k,j\in \zzz}$, define
\bee\label{sum3}
S(f) = \sup_{j\in \mathbf Z} \Big| \sum_{k\in \mathbf Z} a_{k,j} D_k (f)\Big|, 
\eee
whenever $f\in L^1_{loc}(\mathbf R)$. Then $S$ is $0$-local; indeed, if $f$ is supported in a dyadic interval $I_0$ of 
length $2^{-k_0}$ and has vanishing integral, then for any $k<k_0$ and any $J$ of length $2^{-k}$ containing $I_0$ we have that $f$ 
is supported in the left or right half of $J$ on either of which $h_J$ is constant. But if $f$ has 
mean value zero, it follows that $D_k (f)=0$ when $k<k_0$. Therefore the smallest $k$ that appears in the sum in \eqref{sum3} is $k_0$. Also, the sum in \eqref{sum2} for $k=k_0$ 
contains only one term, namely the one corresponding to $I=I_0$.  Finally, the part of the sum in \eqref{sum2} with $k>k_0$ contains only terms which are supported in $I_0$. Thus, $S(f)$ is supported in $I_0$.

Such examples of operators $S$ are inspired by   maximal combinations of martingale difference operators on probability spaces equipped with dyadic filtrations~\cite{GK}.

 \end{document}